\documentclass[1pt]{article}
\usepackage{amsmath, amssymb, amsthm, amsfonts, cases}
\usepackage{mathrsfs}
\usepackage{url}
\usepackage{authblk}
\usepackage[usenames]{color}
\usepackage{geometry}

\newtheorem{lemma}{Lemma}[section]
\newtheorem{theorem}{Theorem}[section]

\newtheorem{definition}{Definition}[section]
\newtheorem{pro}{Problem}[section]
\newtheorem{ass}{Assumption}[section]

\newcommand{\eps}{\varepsilon}

\newcommand{\la}{\langle}
\newcommand{\ra}{\rangle}

\allowdisplaybreaks

 \makeatother
\begin{document}

\title{   Maximum Principle of  Forward-Backward Stochastic Differential System of Mean-Field Type with Observation Noise
 \thanks{This work was supported by the Natural Science Foundation of Zhejiang Province
for Distinguished Young Scholar  (No.LR15A010001),  and the National Natural
Science Foundation of China (No.11471079, 11301177) }}

\date{}

   \author{ Qingxin Meng\thanks{Corresponding author.   E-mail: mqx@zjhu.edu.cn}  \hspace{1cm}
   Qiuhong Shi
 \hspace{1cm}  Maoning Tang
\hspace{1cm}
\\\small{Department of Mathematics, Huzhou University, Zhejiang 313000, China}}

\maketitle
\begin{abstract}

This paper is concerned with the partial information optimal control problem  of mean-field type  under  partial observation,  where the  system is
  given by a controlled mean-field forward-backward stochastic differential equation with correlated noises between the system and the observation, moreover  the observation
coefficients may depend  not only on the
control process and  but also on  its probability distribution.
Under standard assumptions on the coefficients, necessary
and sufficient conditions for optimality of the control problem in
the form of Pontryagin's maximum principles
are established in a unified way.
\end{abstract}


\maketitle

\section{ Introduction}

In recent years, the systems with interacting behavior have attracted
increasing attention in the stochastic control theory. The so-called mean-field models
are designed to study such systems. The history of the mean-field models can be dated back to
the early works of \cite{K1956,M1966}. Since then,
the mean-field models have been found useful to describe the aggregate
behavior of a large number of mutually interacting particles in diverse
areas of physical sciences, such as statistical mechanics, quantum mechanics
and quantum chemistry. Recent interest is to study the stochastic maximum principle
under the mean-field models. Previous works include \cite{AD2011,BDL2011,L2012,MOZ2012,SS2013,DHQ2013,SMS2014},
and references therein.

It is well known that  forward-backward stochastic differential
equations (FBSDEs in short) consists of a forward stochastic
differential equation (SDE in short)  of It\^{o} type and a backward
stochastic differential equation (BSDE in short) of Pardoux-Peng
(for details see \cite{PaPe90},\cite{ElPe}).
FBSDEs are not only encountered in  stochastic optimal control problems
when applying the stochastic maximum principle but also used in
mathematical finance (see Antonelli \cite{Anto}, Duffie and Epstein
\cite{DuEp}, El Karoui, Peng and Quenez \cite{ElPe} for example). It
now becomes more clear that certain important problems in
mathematical economics and mathematical finance, especially in the
optimization problem, can be formulated to be FBSDEs.
 There are two important approaches to the general stochastic optimal
control problem. One is the Bellman dynamic programming principle,
which results in the Hamilton-Jacobi-Bellman equation.  The other
 is the maximum principle.
 Now the maximum principle of forward-backward stochastic systems
 driven by  Brownian motion have been studied extensively in the
 literature. We refer to \cite{ShWz2006, Wu9,Xu95,Yong10}and references therein.

The main contribution of this paper is
that one sufficient (a verification
theorem)  and one necessary optimality conditions  for the existence of optimal controls of FBSDE of mean field type are 
established in a unified way. The
main idea is to get directly a variation formula in terms of the
Hamiltonian and the associated adjoint system which is governed by  a linear mean-field
forward-backward
stochastic differential equation of mean-field  and neither the corresponding Taylor type expansions of 
the state process and the cost functional nor the variational systems will be used.

The paper is organized as follows. In section 2,  the partial information optimal problem of mean-field systems for 
FBSDE is formulated
 and various assumptions used throughout the paper is presented. In Section 3 we establish sufficient and necessary stochastic maximum principles in a unified way.

\section{Formulation of Problem}

Let ${\cal T} : = [0, T]$ denote a
 fixed time interval of finite length,
 i.e., $T<\infty.$ We equip $( \Omega,
{\mathscr F}, {\mathbb P} )$ a complete probability space with a
right-continuous, $\mathbb P-$
complete filtration $\mathbb F=\{\mathscr{F}_t\}_{t\in {\cal T}}$, to be specified below. By $\mathscr{P}$ we
denote the
predictable $\sigma$ field on $\Omega\times [0, T]$ and by $\mathscr %
B(\Lambda)$ the Borel $\sigma$-algebra of any topological space
$\Lambda.$ Furthermore, we assume
that $\mathscr F_T=\mathscr F.$
Denote by $\mathbb E[\cdot]$ be
the expectation taken with respect
to $\mathbb P.$ Let $\{W(t), t \in {\cal T}\}$ and $\{Y(t),t \in {\cal T}\}$
be two one-dimensional
standard Brownian motions
on $( \Omega,
{\mathscr F}, {\mathbb P} )$. Denote by  $%
\{\mathscr{F}^W_t\}_{t\in {\cal T}}$ and $%
\{\mathscr{F}^Y_t\}_{t\in {\cal T}}$ be the natural
filtration generated by $\{W(t), t\in {\cal T}\}$ and $\{Y(t), t\in {\cal T}\},$ respectively. Assume that
$\mathbb F$ is the $\mathbb P-$
augmentation of the natural
filtration generated by
 $\{\mathscr{F}^W_t\}_{t\in {\cal T}}$ and
$\{\mathscr{F}^Y_t\}_{t\in {\cal T}}.$

 Let $E$ be a Euclidean space.
 Denote the inner product in $E$ by
$\langle\cdot, \cdot\rangle,$  the norm in $ E$  by $|\cdot|,$  the
transpose of the matrix or vector $A$ by $A^{\top }$
For a
function $\psi:\mathbb R^n\longrightarrow \mathbb R,$ denote by
$\psi_x$ its gradient. If $\psi: \mathbb R^n\longrightarrow \mathbb R^k$ (with
$k\geq 2),$ then $\psi_x=(\frac{\partial \phi_i}{\partial x_j})$ is
the corresponding $(k\times n)$-Jacobian matrix. By $\mathscr{P}$ we
denote the
predictable $\sigma$ field on $\Omega\times [0, T]$ and by $\mathscr %
B(\Lambda)$ the Borel $\sigma$-algebra of any topological space
$\Lambda.$ In the follows, $K$ represents a generic constant, which
can be different from line to line.

Next we introduce some spaces of random variable and stochastic
 processes. For any $\alpha, \beta\in [1,\infty),$ let
 \begin{enumerate}
\item[$\bullet$]
$M_{\mathscr{F}}^\beta(0,T;E)$: the space of all $E$-valued and ${%
\mathscr{F}}_t$-adapted processes $f=\{f(t,\omega),\ (t,\omega)\in {\cal T}
\times\Omega\}$ satisfying
$
\|f\|_{M_{\mathscr{F}}^\beta(0,T;E)}\triangleq{\left (\mathbb E\bigg[\displaystyle%
\int_0^T|f(t)|^ \beta dt\bigg]\right)^{\frac{1}{\beta}}}<\infty, $
 \item[$\bullet$] $S_{\mathscr{F}}^\beta (0,T;E):$ the space of all $E$-valued and ${
\mathscr{F}}_t$-adapted c\`{a}dl\`{a}g processes $f=\{f(t,\omega),\
(t,\omega)\in {\cal T}\times\Omega\}$ satisfying $
\|f\|_{S_{\mathscr{F}}^\beta(0,T;E)}\triangleq{\left (\mathbb E\bigg[\displaystyle\sup_{t\in {\cal T}}|f(t)|^\beta \bigg]\right)^{\frac
{1}{\beta}}}<+\infty,$
 \item[$\bullet$]$L^\beta (\Omega,{\mathscr{F}},P;E):$ the space of all
$E$-valued random variables $\xi$ on $(\Omega,{\mathscr{F}},P)$
satisfying $ \|\xi\|_{L^\beta(\Omega,{\mathscr{F}},P;E)}\triangleq
\sqrt{\mathbb E|\xi|^\beta}<\infty,$
 \item[$\bullet$] $M_{\mathscr{F}}^\beta(0,T;L^\alpha (0,T; E)):$ the space of all $L^\alpha (0,T; E)$-valued and ${%
\mathscr{F}}_t$-adapted processes $f=\{f(t,\omega),\ (t,\omega)\in[0,T]%
\times\Omega\}$ satisfying $
\|f\|_{\alpha,\beta}\triangleq{\left\{\mathbb E\bigg[\left(\displaystyle
\int_0^T|f(t)|^\alpha
dt\right)^{\frac{\beta}{\alpha}}\bigg]\right\}^{\frac{1}{\beta}}}<\infty. $
\end{enumerate}

In the following, under  partial observations,  we formulate a class of    optimal control problems of mean-field type. Consider the following
 controlled FBSDE of  mean-field type:

\begin{eqnarray} \label{eq:1}
\left\{
\begin{aligned}
dx(t)=&b(t, x(t), u(t),\mathbb E[x(t)], \mathbb E[u(t)])dt+ \sigma_1(t, x(t), u(t), \mathbb E[x(t)], \mathbb E[u(t)]) dW(t)+\sigma _2(t, x(t), u(t),\mathbb E[x(t)], \mathbb E[u(t)]) dW^u(t),
\\
dy(t)=&f(t, x(t), y(t),z_1(t),z_2(t),u(t),
\mathbb E[x(t)],\mathbb E[y(t)],\mathbb E[z_1(t)],\mathbb E[z_2(t)], \mathbb E[u(t)])dt+ z_1(t) dW(t)+  z_2(t) dW^u(t),\\
x(0)=&x_0,\\
y(T)=&\phi(x(T),\mathbb E[x(T)]),
\end{aligned}
\right.
\end{eqnarray}
where
$u(\cdot)$
 is our admissible
 control process
 taking values  in  $U$  being a nonempty convex subset of $\mathbb R^k;$
 $( x(\cdot),  y(\cdot), z_1(\cdot), z_2(\cdot))$, the solution of
 \eqref{eq:4} is the state
 process with initial
 state $x_0$ and $W^u,$ taking
 values in $\mathbb R,$ is a
 stochastic
 process depending on the control
 process $u(\cdot).$
 In the above, $b: {\cal T} \times \Omega \times ({\mathbb R}^n
 \times U)^2  \rightarrow {\mathbb R}^n$,
  $\sigma_1:
{\cal T} \times \Omega \times ({\mathbb R}^n\times U)^2 \rightarrow {\mathbb R}^n$, $\sigma_2: {\cal T} \times \Omega \times ({\mathbb R}^n
 \times U)^2 \rightarrow {\mathbb R}^n $,
 $f: {\cal T} \times \Omega \times ({\mathbb R}^n \times {\mathbb R}^m\times {\mathbb R}^m \times {\mathbb R}^m
 \times U)^2 \rightarrow {\mathbb R}^m,
 \phi: \Omega \times ({\mathbb R}^n)^2  \rightarrow {\mathbb R^m}$
   are given random mapping.

 Suppose that the state
 process $(x(\cdot), y(\cdot), z_1(\cdot), z_2(\cdot))$ cannot be directly.
 Instead, we can observe a related
 process $Y(\cdot)$ is given by the
 following SDE of mean-field type
\begin{eqnarray}\label{eq:2}
\left\{
\begin{aligned}
dY(t)=&h(t, x(t),u(t), \mathbb E[x(t)], \mathbb E[u(t)])dt+ dW^u(t),\\
Y(0)=& 0,
\end{aligned}
\right.
\end{eqnarray}
where
   $h: {\cal T} \times \Omega \times ({\mathbb R}^n
 \times U)^2 \rightarrow {\mathbb R}$
  is a given random mapping.
  In the above equations, $u(\cdot)$ is our admissible control
 process defined as follows.

\begin{definition}
   An  control process $u(\cdot)$ is
   called admissible control process
   if it is  ${\mathscr F}^Y_t$-adapted and valued $U$  such that $$\mathbb E\bigg[\bigg(\int_0^T|u(t)|^2
dt\bigg)^{2}\bigg]<\infty.$$
  Denote by $\cal A$ the set of all admissible controls.
\end{definition}

Now we make the following standard  assumptions
 on the coefficients of the equations
 \eqref{eq:1} and
 \eqref{eq:2}.

\begin{ass}\label{ass:1.1}
(i)For the mapping $\psi=b,\sigma_1,\sigma_2, h,$ it is  ${\mathscr P} \otimes ({\mathscr
B} ({\mathbb R}^n)  \otimes {\mathscr B}
(U))^2 $-measurable such that $\psi(\cdot,0,0,0,0)\in M_{\mathscr{F}}^4(0,T;L^2 (0,T; E))$.  For almost all $(t, \omega)\in {\cal T}
\times \Omega$, the mapping
\begin{eqnarray*}
(x,u,x',u') \rightarrow \psi(t,\omega,x,u,x',u')
\end{eqnarray*}
 is continuous differentiable with respect to $(x,u,x',u')$ with continuous and uniformly bounded derivatives, where $\psi=b, \sigma_1,  \sigma_2$ and $h.$ (ii)The mapping $f$ is  ${\mathscr P} \otimes ({\mathscr
B} ({\mathbb R}^n)\otimes {\mathscr
B} ({\mathbb R}^m)\otimes {\mathscr
B} ({\mathbb R}^m) \otimes {\mathscr
B} ({\mathbb R}^m)  \otimes {\mathscr B}
(U))^2 $-measurable such that
$f(\cdot,0,0,0,0,0,0)\in M_{\mathscr{F}}^4(0,T;L^2 (0,T; E))$. For almost all $(t, \omega)\in {\cal T}
\times \Omega$, the mapping
\begin{eqnarray*}
(x,y, z_1, z_2,u,x',y', z_1', z_2',u') \rightarrow f(t,\omega,x,y,z_1,z_2, u,
x,'y,' z_1', z_2',u')
\end{eqnarray*}
 is continuous differentiable with respect to $(x,y, z_1, z_2,u,x',y', z_1', z_2',u')$ with
appropriate growths. More precisely, there
is  a constant $C
> 0$  such that for  all $(x,y, z_1, z_2,u,x',y', z_1', z_2',u') \in
(\mathbb {R}^n\times \mathbb  R^m\times \mathbb R^m\times \mathbb R^m \times U)^2$ and a.e. $(t, \omega)\in {\cal T} \times \Omega,$
\begin{eqnarray}
\begin{split}
& (1+|x|+|y|+|z_1|+|z_2|+|u|+|x'|+|y'|+|z_1'|
+|z_2'|+|u'|)^{-1}
|f(t,x,y, z_1, z_2,u,x',y', z_1', z_2',u')|
\\&+|f_x(t,x,y, z_1, z_2,u,x',y', z_1', z_2',u')|
+|f_y(t,x,y, z_1, z_2,u,x',y', z_1', z_2',u')|
+|f_{z_1}(t,x,y, z_1, z_2,u,x',y', z_1', z_2',u')|
\\&+|f_{z_2}(t,x,y, z_1, z_2,u,x',y', z_1', z_2',u')|
+|f_u(t,x,y, z_1, z_2,u,x',y', z_1', z_2',u')|
+|f_{x'}(t,x,y, z_1, z_2,u,x',y', z_1', z_2',u')|
\\&+|f_{y'}(t,x,y, z_1, z_2,u,x',y', z_1', z_2',u')|
+|f_{z_1'}(t,x,y, z_1, z_2,u,x',y', z_1', z_2',u')|
+|f_{z_2'}(t,x,y, z_1, z_2,u,x',y', z_1', z_2',u')|
\\&+|f_{u'}(t,x,y, z_1, z_2,u,x',y', z_1', z_2',u')|
\leq C.
\end{split}
\end{eqnarray}
(iii) The mapping
$\phi$ is ${\mathscr F}_T \otimes ({\mathscr B} ({\mathbb
R}^n))^2 $-measurable.  For almost all $(t, \omega)\in [0,T]
\times \Omega$, the mapping

\begin{eqnarray*}
x \rightarrow \phi(\omega,x,x')
\end{eqnarray*}
 is continuous differentiable with respect to $(x,x')$ with
appropriate growths, respectively.
More precisely, there exists a constant $C
> 0$  such that for  all $(x,x')\in (\mathbb {R}^n)^2$ and a.e. $ \omega\in  \Omega,$

\begin{eqnarray}
\begin{split}
& (1+|x|+|x'|)^{-1}
|\phi(x,x')|
+|\phi_x(x,x')|+|\phi_{x'}(x,x')|
\leq C.
\end{split}
\end{eqnarray}

\end{ass}

  Inserting \eqref{eq:2}  \eqref{eq:1}get
  that
\begin{eqnarray} \label{eq:4}
\left\{
\begin{aligned}
dx(t)=&(b-\sigma_2h)(t, x(t),u(t), \mathbb E[x(t)], \mathbb E[u(t)])dt+ \sigma_1(t, x(t),u(t), \mathbb E[x(t)], \mathbb E[u(t)]) dW(t)+\sigma _2(t, x(t),u(t), \mathbb E[x(t)], \mathbb E[u(t)]) dY(t),
\\
dy(t)=&(f(t, x(t), y(t),z_1(t),z_2(t),u(t),
\mathbb E[x(t)],\mathbb E[y(t)],\mathbb E[z_1(t)],\mathbb E[z_2(t)], \mathbb E[u(t)])
-z_2(t)h(t, x(t),u(t), \mathbb E[x(t)], \mathbb E[u(t)]))dt
\\&+ z_1(t) dW(t)+  z_2(t) dY(t),\\
x(0)=&x,\\
y(T)=&\Phi(x(T),
\mathbb E[x(T)]).
\end{aligned}
\right.
\end{eqnarray}
For any
admissible control $u(\cdot)\in \cal A,$ by Assumption \ref{ass:1.1},
it is easy to have the following
important result on \eqref{eq:4}.

\begin{lemma}\label{lem:3.3}
   Suppose that Assumption \ref{ass:1.1} hold. Then associated wtih
  admissible control $u(\cdot)\in
  \cal A,$   \eqref{eq:4} have  a unique strong
  solution $( x(\cdot),  y(\cdot), z_1(\cdot), z_2(\cdot))\in S_{\mathscr{F}}^4 (0,T;\mathbb R^{n})\times S_{\mathscr{F}}^4 (0,T;\mathbb R^{m})
  \times M_{\mathscr{F}}^2(0,T;L^2 (0,T; \mathbb R^m))\times M_{\mathscr{F}}^2(0,T;L^2 (0,T; \mathbb R^m)).$
Moreover,  the following basic
estimate holds:

\begin{eqnarray}\label{eq:1.8}
\begin{split}
 {\mathbb E} \bigg [ \sup_{t\in \cal T} | x(t) |^4 \bigg
] +{\mathbb E} \bigg [ \sup_{t\in \cal T} | y(t) |^4 \bigg
]+{\mathbb E} \bigg [ \bigg(\int_0^T | z_1(t) |^2 dt\bigg)^{2}
+{\mathbb E} \bigg [ \bigg(\int_0^T | z_2(t) |^2 dt\bigg)^{2}\bigg
]&\leq& K\bigg \{1+|x|^4+\mathbb E\bigg[\Big(\int_0^T |u(t)|^2dt\Big)^{2}\bigg] \bigg\}.
\end{split}
\end{eqnarray}
Moreover, if $(\bar x(\cdot), \bar y(\cdot), \bar z_1(\cdot),
\bar z_2(\cdot))$ is the unique strong solution associated with another admissible control  $ \bar u(\cdot)\in \cal A,$ we have
\begin{eqnarray}\label{eq:1.15}
\begin{split}
&{\mathbb E} \bigg [ \sup_{t\in \cal T} | x (t) - \bar x(t) |^4
\bigg ]+{\mathbb E} \bigg [ \sup_{t\in \cal T} | y (t) - \bar y(t) |^4
\bigg ]+{\mathbb E} \bigg [ \bigg(\int_0^T | z_1(t)-\bar z_1(t) |^2 dt\bigg)^{2}\bigg
]+{\mathbb E} \bigg [ \bigg(\int_0^T | z_2(t)-\bar z_2(t) |^2 dt\bigg)^{2}\bigg
]
\\&\leq  K {\mathbb E} \bigg [ \int_0^T| u(t)- \bar u (t)
|^2 dt\bigg ]^{2}.
\end{split}
\end{eqnarray}
\end{lemma}
\begin{proof}
  The proof can be proved
  similar to   Proposition 2.1 in
  \cite{Mou} and Lemma 2
  in \cite{Xu95}.
\end{proof}

For any  given admissible control
$u(\cdot)\in \cal A$  and the corresponding  strong solution
 $( x^u(\cdot),  y^u(\cdot), z_1^u(\cdot), z_2^u(\cdot))$ of  the equation
 \eqref{eq:4},  define a stochastic
 process by
\begin{eqnarray}\label{eq:7}
\begin{split}
  \rho^u(t)
  =\displaystyle
  \exp^{\bigg\{\displaystyle\int_0^t h(s,x^u(s),u(s),\mathbb E[x^u(s)], \mathbb E[u(s)])dY(s)
  -\frac{1}{2} h^2(s,x^u(s),u(s), \mathbb E[x^u(s)], \mathbb E [u(s)])ds\bigg\}},
  \end{split}
\end{eqnarray}
which is  the solution of  the   SDE
\begin{eqnarray} \label{eq:8}
  \left\{
\begin{aligned}
  d \rho^u(t)=&  \rho^u(t) h(s,x^u(s),u(s),\mathbb E[x^u(s)], \mathbb E[u(s)])dY(s)\\
  \rho^u(0)=&1.
\end{aligned}
\right.
\end{eqnarray}

 The following basic result is  on the stochastic process
 $\rho^u(\cdot),$
 \begin{lemma}\label{lem:3.4}
   Suppose that Assumption \ref{ass:1.1} holds. Then for any $u(\cdot)\in \cal A$ and  any $\alpha \geq 2,$ it follows that
\begin{eqnarray}\label{eq:1.9}
\begin{split}
 {\mathbb E} \bigg [ \sup_{{t\in
 \cal T}} | \rho^u(t) |^\alpha \bigg
] \leq K.
\end{split}
\end{eqnarray}
Further, if $ \bar \rho(\cdot)$
is the solution of  \eqref{eq:8}
associated with another
 admissible control $ \bar u(\cdot)\in \cal A,$  we get that
\begin{eqnarray}\label{eq:1.15}
{\mathbb E} \bigg [ \sup_{t\in \cal T} |
\rho^u (t) - \bar \rho (t) |^2
\bigg ]  \leq  K \bigg\{{\mathbb E} \bigg [ \int_0^T| u(t)- \bar u (t)
|^2 dt\bigg ]^{{2}}\bigg\}^{\frac{1}{2}}.
\end{eqnarray}
\end{lemma}

\begin{proof}
  The proof can be proved similarly to the
  proof of Proposition 2.1 in
  \cite{Mou}.
\end{proof}

Under Assumption \ref{ass:1.1},
$\rho^u(\cdot)$ is
an  $( \Omega,
{\mathscr F}, \{\mathscr{F}_t\}_{t\in {\cal T}}, {\mathbb P} )-$
martingale.  we thus can  introduce a new probability measure $\mathbb P^u$ on $(\Omega, \mathscr F)$ by
\begin{eqnarray}
  d\mathbb P^u=\rho^u(1)d\mathbb P.
\end{eqnarray}
 Then using Girsanov's theorem and \eqref{eq:2}, $(W(\cdot),W^u(\cdot))$ is an
$\mathbb R^2$-valued standard Brownian motion on the new probability
space $(\Omega, \mathscr F, \{\mathscr{F}_t\}_{0\leq t\leq T},\mathbb P^u).$
So $(\mathbb P^u, x^u(\cdot), y^u(\cdot),
z^u_1(\cdot), z^u_2(\cdot), \rho^u(\cdot),
 W(\cdot), W^u(\cdot))$ is a weak
solution on $(\Omega, \mathscr F, \{\mathscr{F}_t\}_{t\in \cal
T})$ of  \eqref{eq:1} and
\eqref{eq:2}.

 Give the cost functional  by
\begin{eqnarray}\label{eq:13}
  \begin{split}
    J(u(\cdot)=\mathbb E^u\bigg[\int_0^Tl(t, x(t), y(t),z_1(t),z_2(t),u(t),
\mathbb E[x(t)],\mathbb E[y(t)],\mathbb E[z_1(t)],\mathbb E[z_2(t)], \mathbb E[u(t)])dt+ \Phi(x(T),\mathbb E[x(T)])+\gamma(y(0))\bigg].
  \end{split}
\end{eqnarray}
 where $\mathbb E^u$ stands for the
 mathematical expectation on $(\Omega, \mathscr F, \{\mathscr{F}_t\}_{0\leq
t\leq T},\mathbb P^u)$ and  the following assumption on $l:
{\cal T} \times \Omega
\times ({\mathbb R}^n \times {\mathbb R}^m
\times {\mathbb R}^m\times {\mathbb R}^m
 \times U )^2\rightarrow {\mathbb R},$ $\Phi: \Omega \times ({\mathbb R}^n)^2  \rightarrow {\mathbb R}$
 and  $\gamma: \Omega \times {\mathbb R}^m \rightarrow {\mathbb R}$
  will be needed:

 \begin{ass}\label{ass:1.2}
 $\Phi$ is ${\mathscr F}_T \otimes ({\mathscr B} ({\mathbb
R}^n))^2 $-measurable, and $\gamma$ is ${
\mathscr F}_0 \otimes {\mathscr B} ({\mathbb
R}^n) $-measurable,$l$ is ${\mathscr P} \otimes ({\mathscr
B} ({\mathbb R}^n) \otimes {\mathscr B} ({\mathbb R}^m)\otimes {\mathscr B} ({\mathbb R}^m) \otimes {\mathscr B} ({\mathbb R}^m)  \otimes {\mathscr B}
(U))^2 $-measurable. For almost all $(t, \omega)\in [0,T]
\times \Omega$, the mappings $l,\Phi,\gamma$
are continuous differentiable with respect to $(x,y, z_1, z_2,u,x',y', z_1', z_2',u')$ with
appropriate growths, respectively.
More precisely, for all $(x,y, z_1, z_2,u,x',y', z_1', z_2',u') \in (\mathbb {R}^n\times \mathbb R^m
\times \mathbb R^m\times \mathbb R^m\times U)^2$ and a.e. $(t, \omega)\in [0,T]
\times \Omega,$  it follows that
\begin{eqnarray*}
\left\{
\begin{aligned}
&
(1+|x|+|y|+|z_1|+|z_2|+|u|
+|x'|+|y'|+|z_1'|+|z_2'|+|u'|)^{-1} (|l_x(t,x,y,z_1,z_2,u,x',y', z_1', z_2',u')|
\\&\quad\quad+|l_y(t,x,y,z_1,z_2,u,x',y', z_1', z_2',u')|+
|l_{z_1}(t,x,y,z_1,z_2,u,x',y', z_1', z_2',u')|
+|l_{z_2}(t,x,y,z_1,z_2,u,x',y', z_1', z_2',u')|
\\&\quad\quad+|l_u(t,x,y,z_1,z_2,u,x',y', z_1', z_2',u')|+|l_{x'}(t,x,y,z_1,z_2,u,x',y', z_1', z_2',u')|+|l_{y'}(t,x,y,z_1,z_2,u,x',y', z_1', z_2',u')|
\\&\quad\quad+
|l_{z_1'}(t,x,y,z_1,z_2,u,x',y', z_1', z_2',u')|+|l_{z_2'}(t,x,y,z_1,z_2,u,x',y', z_1', z_2',u')|
+|l_{u'}(t,x,y,z_1,z_2,u,x',y', z_1', z_2',u')|)
\\&\quad\quad+(1+|x|^2+|y|^2+|z_1|^2
+|z_2|^2
+|u|^2+|x'|^2+|y'|^2+|z_1'|^2
+|z_2'|^2
+|u'|^2)^{-1}|l(t,x,y,z_1,z_2,u,x',y', z_1', z_2',u')|
 \leq C;
\\
& (1+|x|^2)^{-1}|\Phi(x,x')| +(1+|x|)^{-1}|\Phi_x(x,x')|
+(1+|x'|)^{-1}|\Phi_{x'}(x,x')|\leq
C;
\\
& (1+|y|^2)^{-1}|\gamma(y)| +(1+|y|)^{-1}|\gamma_y(y)|\leq
C.
\end{aligned}
\right.
\end{eqnarray*}
\end{ass}
 Under  Assumption \ref{ass:1.1} and
 \ref{ass:1.2},
 by the estimates \eqref{eq:1.8} and
 \eqref{eq:1.9},  it is easy to check that
the cost functional is well-defined.

 Now we pose  an optimal control
  problem of forward-backward
  stochastic differential systems
  with partial information in its weak formulation,
 i.e., with changing
 the reference probability space $(\Omega, \mathscr F, \{\mathscr{F}_t\}_{0\leq
t\leq T},\mathbb P^u),$ as follows.

\begin{pro}
\label{pro:1.1} Seek $\bar{u}(\cdot)\in \cal$ such that
\begin{equation*}  \label{eq:b7}
J(\bar{u}(\cdot))=\displaystyle\inf_{u(\cdot)
\in \cal A}J(u(\cdot)),
\end{equation*}
subject to  \eqref{eq:1}, \eqref{eq:2}
 and \eqref{eq:13}.

\end{pro}
 Using  Bayes' formula, it is easy to check that we can rewrite
 the cost functional \eqref{eq:13} as
\begin{equation}\label{eq:15}
\begin{split}
J(u(\cdot ))=& \mathbb E\displaystyle\bigg[%
\int_{0}^{T}\rho^u(t)l(t, x(t), y(t),z_1(t),z_2(t),u(t),
\mathbb E[x(t)],\mathbb E[y(t)],\mathbb E[z_1(t)],\mathbb E[z_2(t)], \mathbb E[u(t)])dt +\rho^u(T)\Phi(x(T),\mathbb E[x(T)])
+\gamma(y(0))\bigg].
\end{split}
\end{equation}
  Therefore, Problem \ref{pro:1.1} can be translated
  into  the following  equivalent optimal control problem in
  its strong formulation, i.e., without changing the  reference
probability space $(\Omega, \mathscr F, \{\mathscr{F}_t\}_{0\leq t\leq T},\mathbb P).$  Here we will regard $\rho^u(\cdot)$  as
an additional  state process besides the
state process $(x^u(\cdot), y^u(\cdot),
z^u_1(\cdot), z^u_2(\cdot)).$

\begin{pro}
\label{pro:1.2} Seek $\bar{u}(\cdot)\in \cal A$ such that
\begin{equation*}  \label{eq:b7}
J(\bar{u}(\cdot))=\displaystyle\inf_{u(\cdot)
\in \cal A}J(u(\cdot)),
\end{equation*}
 subject to \eqref{eq:15}
 and  the following
 state equation
\begin{equation}
\displaystyle\left\{
\begin{array}{lll}
dx(t)=&(b-\sigma_2h)(t, x(t), u(t),\mathbb  E[x(t)],\mathbb E[u(t)])dt+ \sigma_1(t, x(t), u(t), \mathbb E[x(t)],\mathbb E[u(t)]) dW(t)\\&\quad+\sigma _2(t, x(t), u(t), \mathbb E[x(t)],\mathbb E[u(t)]) dY(t),
\\
dy(t)=&(f(t, x(t), y(t),z_1(t),z_2(t),u(t),
\mathbb E[x(t)],\mathbb E[y(t)],\mathbb E[z_1(t)],\mathbb E[z_2(t)], \mathbb E[u(t)])
-z_2(t)h(t, x(t), u(t),\mathbb  E[x(t)],\mathbb E[u(t)])dt\\&+ z_1(t) dW(t)
+  z_2(t) dY(t),\\
d \rho(t)=&  \rho(t) h(s,x(s), u(s), \mathbb E[x^u(s)],\mathbb E[ u(s)])dY(s),\\
  \rho^u(0)=&1,
\\
x(0)=&x,\\
y(T)=&\Phi(x(T),\mathbb E[x(T)]).
\end{array}%
\right.  \label{eq:3.7}
\end{equation}
\end{pro}
Here  we call $\bar{u}(\cdot)\in \cal A$ satisfying above  an optimal
control process of Problem \ref{pro:1.2} and the corresponding state
process $(\bar x(\cdot),
\bar y(\cdot),
\bar z_1(\cdot), \\\bar z_2(\cdot),
\bar \rho(\cdot))$  the optimal
state process. Correspondingly $(\bar{u}(\cdot);\bar x(\cdot),
\bar y(\cdot),
\bar z_1(\cdot), \bar z_2(\cdot),
\bar \rho(\cdot))$ is said to be  an optimal pair of Problem \ref{pro:1.2}.

\section{A Variation Formulation for
the Cost Functional}

In this section, we will establish
a variation formulation for the cost
functional by using the
Hamiltonian and adjoint process.

We first define  the Hamiltonian by
 ${\cal H}: \Omega \times {\cal T} \times (\mathbb R^n \times \mathbb R^m \times \mathbb R^{m}
  \times \mathbb R^{m}\times U)^2
  \times \mathbb R^{m}\times
\mathbb R^n\times \mathbb R^n
\times \mathbb R^n\times \mathbb R\rightarrow \mathbb R$  as follows:
\begin{eqnarray}\label{eq6}
&& { H} (t, x, y, z_1,z_2,u, x', y', z_1',z_2', u',k, p,q_1, q_2, R_2) \nonumber \\
&& = l (t, x, y, z_1,z_2,u, x', y', z_1',z_2', u')
+ \langle b (t,x,u,x',u'),  p \rangle
+ \langle \sigma_1(t, x,u,x',u'), q_1\rangle + \langle \sigma_2 (t, x,u,x',u'), q_2\rangle \nonumber
\\&&\quad +
\langle f(t, x, y, z_1,z_2, u, x', y', z_1',z_2',u'), k\rangle  +\langle R_2, h(t,x,u,x',u')\rangle \ .
\end{eqnarray}

For any given admissible control pair $(\bar u(\cdot); \bar x(\cdot),
 \bar y(\cdot), \bar z_1(\cdot),
 \bar z_2(\cdot)), $  we define the
 the corresponding adjoint process as the
 solution as the solution of the following FBSDE:
\begin{numcases}{}\label{eq:18}
\begin{split}
d\bar r(t)&=- l(t,\bar\Theta(t), \mathbb E[\bar \Theta(t)],\bar u(t),\mathbb E[\bar u(t))dt
 +\bar R_1\left(
t\right)  dW\left(  t\right) +{\bar R}_{2}\left( t\right) dW^{
\bar u}\left( t\right),
\\
d\bar p\left(  t\right)
 &=-\bigg[{\cal H}_x (t,\bar\Theta(t),\bar u(t), \mathbb E[\bar \Theta(t)],\mathbb E[\bar u(t))],
 \bar\Lambda(t), \bar R_2(t))+\frac{1}{\bar\rho(t)}\mathbb E^{\bar u}[{\cal H}_{x'} (t,\bar\Theta(t),\bar u(t), \mathbb E[\bar \Theta(t)],\mathbb E[\bar u(t))],
 \bar\Lambda(t), \bar R_2(t))]\bigg]dt
 \\&\quad+\bar q_1
\left(  t\right)  dW\left(  t\right)  +{\bar q}_{2}\left( t\right) dW^{
\bar u}\left( t\right),
\\
d\bar k\left(  t\right)  &=-\bigg[{\cal H}_y (t,\bar\Theta(t),\bar u(t), \mathbb E[\bar \Theta(t)],\mathbb E[\bar u(t))],
 \bar\Lambda(t), \bar R_2(t))+\frac{1}{\bar\rho(t)}\mathbb E^{\bar u}[{\cal H}_{y'} (t,\bar\Theta(t), \bar u(t),\mathbb E[\bar \Theta(t)],\mathbb E[\bar u(t))],
 \bar\Lambda(t), \bar R_2(t))]\bigg]dt
\\&\quad-
\bigg[{\cal H}_{z_1} (t,\bar\Theta(t),\bar u(t), \mathbb E[\bar \Theta(t)],\mathbb E[\bar u(t))],
 \bar\Lambda(t), \bar R_2(t))+\frac{1}{\bar\rho(t)}\mathbb E^{\bar u}[{\cal H}_{z_1'} (t,\bar\Theta(t),\bar u(t), \mathbb E[\bar \Theta(t)],\mathbb E[\bar u(t))],
 \bar\Lambda(t), \bar R_2(t))]\bigg]  dW\left(  t\right)
  \\&\quad-\bigg[{\cal H}_{z_2} (t,\bar\Theta(t),\bar u(t), \mathbb E[\bar \Theta(t)],\mathbb E[\bar u(t))],
 \bar\Lambda(t), \bar R_2(t))+\frac{1}{\bar\rho(t)}\mathbb E^{\bar u}[{\cal H}_{z_2'} (t,\bar\Theta(t),\bar u(t), \mathbb E[\bar \Theta(t)],\mathbb E[\bar u(t))],
 \bar\Lambda(t), \bar R_2(t))]\bigg]dW^{
\bar u}\left( t\right),
\\ \bar p(T)&=\Phi_x(\bar x(T),\mathbb E[\bar x (T)])+\frac{1}{\bar \rho(T)}\mathbb E^{\bar u} \left[\Phi_x(\bar x(T),\mathbb E [\bar x(T)])\right]\\&\quad-\bigg[\phi_x^*(\bar x(T),\mathbb E [\bar x(T)])\bar k(T)+\frac{1}{\bar \rho(T)}\mathbb E^{\bar u} \left[\phi_x^*(\bar x(T),\mathbb E [\bar x(T)] )\bar k(T)\right]\bigg],
\\ \bar r(T)&=\Phi(\bar x(T),\mathbb E [\bar x(T)])+\frac{1}{\bar \rho(T)}\mathbb E^{\bar u} \left[\Phi(\bar x(T),\mathbb E [\bar x(T)])\right],
\\ \bar k(0)&=-\gamma_y(\bar y(0)).
\end{split}
\end{numcases}
Here the following short hand notation
have been used:
\begin{eqnarray}
  \begin{split}
   &\bar\Theta(t):=(\bar x(t),
   \bar y(t), \bar z_1(t),\bar z_2(\cdot)),
  \\ &\mathbb E\big[\bar\Theta(t)\big]:=( \mathbb E[\bar x(t)],
   \mathbb E[\bar y(t)],
   \mathbb E[\bar z_1(t)],
   \mathbb E[\bar z_2(\cdot))]),
\\&\bar\Lambda(t):=(\bar k(t), \bar p(t), \bar q_1(t), \bar q_2(\cdot)),
\\&\bar\Gamma(t):=(\bar r(t),\bar R_1(t), \bar R_2(\cdot)).
  \end{split}
\end{eqnarray}
\begin{eqnarray}\label{eq:19}
\begin{split}
  &{\cal H}_a( t,\bar\Theta(t),\bar u(t), \mathbb E[\bar \Theta(t)],\mathbb E[\bar u(t)),
 \bar\Lambda(t), \bar R_2(t))
 \\=&{ {H}}_{a}\left(   t,\bar\Theta(t),\bar u(t), \mathbb E[\bar \Theta(t)],\mathbb E[\bar u(t)),
 \bar\Lambda(t),  \bar R_2(t)-\sigma_2^*(t, \bar x(t),\bar u(t),\mathbb E[\bar x(t)],
\mathbb E[\bar u(t)] )\bar p(t)-\bar z_2 ^*(t)
  \bar k(t) \right),
  \end{split}
\end{eqnarray}
where $a=x, y,z_1, z_2,u,x', y',z'_1, z_2',u'.$
 Here the FBSDE  \eqref{eq:18}
is said to be  the adjoint equation whose solution
consists of  an 7-tuple process $(\bar p(\cdot),\bar q_1(\cdot), \bar { q}_2(\cdot),\bar k(\cdot),\bar r(\cdot),\bar R_1(\cdot),\bar R_2(\cdot ) ).$
 In view of  Assumptions \ref{ass:1.1} and
\ref{ass:1.2},  from Lemma 2
  in \cite{Xu95} and  Proposition 2.1 in
  \cite{Mou},  the adjoint equation \eqref{eq:18} has
a unique strong solution $(\bar p(\cdot),\bar q_1(\cdot), \bar { q}_2(\cdot),\bar k(\cdot),\bar r(\cdot),\bar R_1(\cdot),\bar R_2(\cdot ) )\in S_{\mathscr{F}}^4(0,T;
\mathbb R^n)\times
M_{\mathscr{F}}^2(0,T;L^2 (0,T; \mathbb R^n)) \times M_{\mathscr{F}}^2(0,T;L^2 (0,T; \mathbb R^n))\times S_{\mathscr{F}}^4(0,T;
\mathbb R^m)\times S_{\mathscr{F}}^4(0,T;
\mathbb R)\times M_{\mathscr{F}}^2(0,T;L^2 (0,T; \mathbb R))\times M_{\mathscr{F}}^2(0,T;L^2 (0,T; \mathbb R^n)),
$  also said to be  the adjoint process corresponding to
the admissible pair $(\bar{u}(\cdot);\bar x(\cdot),
\bar y(\cdot),
\bar z_1(\cdot), \bar z_2(\cdot),
\bar \rho(\cdot))$.

 For any two  admissible pairs $({u}(\cdot);\Theta^u(\cdot),
\rho^u(\cdot))=({u}(\cdot);x^u(\cdot),
 y^u(\cdot),
 z^u_1(\cdot),  z^u_2(\cdot),
\rho^u(\cdot))$ and $({u}(\cdot);\bar\Theta(\cdot),
\bar\rho(\cdot))=(\bar{u}(\cdot);\bar x(\cdot),
\bar y(\cdot),
\bar z_1(\cdot), \\\bar z_2(\cdot),
\bar \rho(\cdot))$,  we give a presentation for the difference $J (u (\cdot)) - J (\bar u (\cdot))$
in terms of the adjoint process $(\bar \Lambda(\cdot),\bar \Gamma(\cdot) )=(\bar k(\cdot),\bar p(\cdot),\bar q_1(\cdot), \bar { q}_2(\cdot),\\\bar r(\cdot),\bar R_1(\cdot),\bar R_2(\cdot ) )$ and the Hamiltonian ${\cal H}$ and
as well as other relevant expressions.

In the following, to simplify  the  notation , we denote by :
\begin{eqnarray}\label{eq:3.16}
\left\{
\begin{aligned}
& \gamma^u (0)
= \gamma( y^u (0)) \ ,
 \bar \gamma(0) =
 \gamma( \bar y (0)), \
 \\& \phi^u (T)
= \phi( x^u (T),\mathbb E[x^u(T)]),
 \bar \phi(T) =  \phi( \bar x (T),\mathbb
 E[\bar x(T)]) \ ,
\\& \Phi^u (T)
= \Phi( x^u (T),\mathbb E[x^u(T)]),
 \bar \Phi(T) =  \Phi( \bar x (T),
 \mathbb E[\bar x(T)]),\\
& \alpha^u (t) =
 \alpha ( t, x^u (t),u (t), \mathbb E[x^u(t)],\mathbb E[u(t)]),\\
& \bar \alpha (t) = \alpha( t,
 \bar x (t),\bar
u(t),\mathbb E[\bar x(t)],\mathbb E[\bar u(t)] ), \quad \alpha =  b, \sigma_1,
\sigma_2, h,
\\
& \beta^u (t) = \beta ( t, \Theta^u (t),
u (t),\mathbb E[\Theta^u(t)],
\mathbb E[u(t)] ),\\
& \bar \beta (t) = \beta( t, \bar\Theta (t),\bar u(t),
\mathbb E[\bar\Theta(t)],
\mathbb E[\bar u(t)] ), \quad \beta =  f, l,\\
& \bar {\cal H} (t) = {\cal H}( t,\bar\Theta(t),\bar u(t), \mathbb E[\bar \Theta(t)],\mathbb E[\bar u(t))],
 \bar\Lambda(t), \bar R_2(t)).
\end{aligned}
\right.
\end{eqnarray}

\begin{lemma}\label{lem4}
Suppose that Assumptions \ref{ass:1.1} and \ref{ass:1.2} holds. Using the abbreviation \eqref{eq:19} and
\eqref{eq:3.16},  it follows that
\begin{eqnarray}\label{eq:21}
&&J (u (\cdot)) - J (\bar u(\cdot))\nonumber
\\ &=& {\mathbb E} ^{\bar u}
\bigg[\int_0^T \bigg \{ { \cal H}(t,\Theta^u(t),u(t), \mathbb E[\Theta^u(t)], \mathbb E[u(t)],
 \bar\Lambda(t),\bar R_2(t)) - {\bar{\cal H}} (t)\nonumber
   \\&&\quad\quad\quad- \big < {\bar{\cal H}_{x'}} (t)+\frac{1}{\bar\rho(t)}\mathbb E^{\bar u}[{\cal H}_{x'} (t)],x^u (t) - \bar
x (t)
\big>\nonumber
\\&&\quad\quad\quad- \big <{ \bar{\cal H}}_y (t)+\frac{1}{\bar\rho(t)}\mathbb E^{\bar u}[{\cal H}_{y'} (t)],
y^u (t) - \bar
y (t)\big>
\\&&\quad\quad\quad- \big <{ \bar{\cal H}}_{z_1} (t)+\frac{1}{\bar\rho(t)}\mathbb E^{\bar u}[{\cal H}_{z_1'} (t)],
z^u_1 (t) - \bar
z_1 (t)\big>\nonumber
\\&&\quad\quad\quad- \big <{ \bar{\cal H}}_{z_2} (t)+\frac{1}{\bar\rho(t)}\mathbb E^{\bar u}[{\cal H}_{z_2'} (t)],
z^u_2 (t) - \bar
z_2 (t)\big>
\nonumber
\\&&\quad\quad\quad-
 \langle (\sigma_2^u(t)-\bar{\sigma}_2(t))
 (h^u(t)-\bar
 h(t)), \bar p(t)\rangle\nonumber
 \\&&\quad\quad\quad-
 \langle (z_2^u(t)-\bar{z}_2(t))
 (h^u(t)-\bar
 h(t)), \bar k(t)\rangle\bigg\} dt\bigg]
\nonumber \\
&&\quad\quad\quad + {\mathbb E}^{\bar u} \big [ \Phi^u(T)
 - \bar \Phi (T) -  \la
x^u (T) - \bar x (T), \bar\Phi_x(T)+\frac{1}{\bar \rho(T)}\mathbb E^{\bar u} \left[\bar\Phi_{x'}(T)\right]
\ra \big ]
\nonumber \\
&&\quad\quad\quad - {\mathbb E}^{\bar u}
\big [ \la \phi^u (T) - \bar \phi (T), \bar k(T)\ra - \langle
\bar \phi_x^*(T)\bar k(T)+\frac{1}{\bar \rho(T)}\mathbb E^{\bar u} \left[\bar\phi_{x'}^*(T)\bar k(T)\right],
x^u (T) - \bar x (T) \rangle \big ]
\nonumber \\
&&\quad\quad\quad + {\mathbb E}
 \big [ \gamma^u(0) - \bar \gamma (0)
 - \la
y^u (0)
- \bar y (0), \bar \gamma_y (0)  \ra\big ]
\nonumber
\\&&\quad\quad\quad+\mathbb E\bigg[\int_0^T \bar R_2(t)(\rho^u(t)-\bar\rho(t))(h^u(t)-\bar
 h(t))dt\bigg] \nonumber
\\&&\quad\quad\quad+\mathbb E\bigg[\int_0^T(l^u(t)-\bar l(t))( \rho^u(t)-\bar
\rho(t))dt\bigg]\nonumber
\\&&\quad\quad\quad+\mathbb E \bigg[ ( \rho^u (T) - \bar \rho(T)) (\Phi^u (T) -
\bar \Phi (T))\bigg].
\end{eqnarray}
\end{lemma}

\begin{proof}

In view of   the definition of the cost function $J(u(\cdot)),$
we get that

\begin{eqnarray}\label{eq:26}
\begin{split}
 J(u(\cdot))-J(\bar u(\cdot))=&
 \mathbb E^u\bigg[\int_0^Tl^u(t)dt+ \Phi^u(T)+\gamma^u(0)\bigg]-\mathbb E^{\bar u}\bigg[\int_0^T\bar l(t)dt+ \bar\Phi(T)+\bar\gamma(0)\bigg]
 \\=& \mathbb E\bigg[\int_0^T (\rho^u(t)l^u(t)-\bar\rho(t)\bar l(t))dt\bigg]+ \mathbb E[\rho^u(T)\Phi^u(T)-\bar \rho (t)\bar \Phi(T)]
 +\mathbb E[\gamma^u(0)-\bar \gamma(0)]
 \\=& \mathbb E^{\bar u}
 \bigg[\int_0^T [l^u(t)-\bar l(t)]dt
 \bigg]+ \mathbb E^{\bar u}[\Phi^u(T)-\bar \Phi(T)]+ \mathbb E\bigg[\int_0^T (\rho^u(t)-\bar\rho(t)) l^u(t)dt\bigg]
\\&+ \mathbb E[(\rho^u(T)-\bar \rho (T))\Phi^u(T)]
 +\mathbb E[\gamma^u(0)-\bar \gamma(0)].
\end{split}
\end{eqnarray}

On the other hand, by \eqref{eq:18},
 it follows that
$(\bar p(\cdot),\bar q_1(\cdot),\bar q_2(\cdot),\bar k(\cdot) )$ satisfies the
following   FBSDE

\begin{numcases}{}\label{eq:3.3}
\begin{split}
d\bar p\left(  t\right)
 &=-\bigg[{ \bar{\cal H}}_x (t)+\frac{1}{\bar\rho(t)}\mathbb E^{\bar u}[{\cal H}_{x'} (t)]\bigg]dt
 +\bar q_1
\left(  t\right)  dW\left(  t\right)  +{\bar q}_{2}\left( t\right) dW^{
\bar u}\left( t\right),
\\
d\bar k\left(  t\right)  &=-\bigg[{ \bar{\cal H}}_y (t)+\frac{1}{\bar\rho(t)}\mathbb E^{\bar u}[{\cal H}_{y'} (t)]\bigg]dt-
\bigg[{ \bar{\cal H}}_{z_1} (t)+\frac{1}{\bar\rho(t)}\mathbb E^{\bar u}[{\cal H}_{z_1'} (t)]\bigg] dW\left(  t\right)
  \\&\quad\quad-\bigg[{ \bar{\cal H}}_{z_2} (t)+\frac{1}{\bar\rho(t)}\mathbb E^{\bar u}[{\cal H}_{z_2'} (t)]\bigg]dW^{
\bar u}\left( t\right),
\\ \bar p(T)&=\bar\Phi_x(T)
+\frac{1}{\bar \rho(T)}\mathbb E^{\bar u} \left[\bar\Phi_{x'}(T)\right]
-\bigg[\bar \phi_x^*(T)\bar k(T)+\frac{1}{\bar \rho(T)}\mathbb E^{\bar u} \left[\bar \phi_{x'}^*(T)\bar k(T)\right]\bigg],
\\ \bar k(0)&=-\bar \gamma_y(0),
\end{split}
\end{numcases}
and  $(\bar r(\cdot),
\bar R_1(\cdot),
\bar R_2(\cdot))$  solves the following BSDE

\begin{numcases}{}\label{eq:3.3}
\begin{split}
d\bar r(t)&=-[\bar l(t)+\bar R_2(t)\bar h(t)]dt+\bar R_1\left(
t\right) dW\left(  t\right) +{\bar R}_{2}\left( t\right) dY(t),
\\ \bar r(T)&=\bar\Phi(T).
\end{split}
\end{numcases}
Moreover, by \eqref{eq:2}, it is easy to check that the $(x^u(\cdot), y^u(\cdot), z^u_1(\cdot),z^u_2(\cdot))$ satisfies the following FBSDE:

\begin{eqnarray}
\left\{
\begin{aligned}
dx(t)=&\big[b^{u}(t)+\sigma _2^u(t)(\bar  h(t)-h^u(t))\big]dt+
\sigma^u_1(t) dW(t)+\sigma _2^u(t) dW^{\bar u}(t)
\\
dy(t)=&\big[f^u(t)+z _2(t)( \bar h(t)- h^u(t))\big]dt+ z_1(t) dW(t)+  z_2(t) dW^{\bar u}(t)\\
x(0)=&x,\\
y(T)=&\phi^u(T)-\bar\phi(T)
\end{aligned}
\right.
\end{eqnarray}
Therefore $(x^u(t)-\bar x(t),
y^u(t)-\bar y(t), z^u_1(t)-\bar z_1(t),
z^u_2(t)-\bar z_2(t))$ solves the following FBSDE:

\begin{eqnarray}
\left\{
\begin{aligned}
dx(t)-\bar x(t)=&\big[b^u(t)-\bar b(t)+\sigma _2^u(t)( h^{\bar u}(t)-h^u(t))\big]dt+ [\sigma^u_1(t)-\bar\sigma_1(t)] dW(t)+[\sigma _2^u(t)-\bar\sigma _2(t)]dW^{\bar u}(t)
\\
dy(t)-\bar y(t)=&\big[f^u(t)-\bar f(t) +z _2(t)( \bar h(t)-h^u(t))\big]dt+ [z_1(t)-\bar z_1(t)] dW(t)+  [z_2(t)-\bar z_2(t)] dW^{\bar u}(t)\\
x(0)-\bar x(0)=&0,\\
y(T)-\bar y(T)=&\phi^u(T)
-\bar\phi(T).
\end{aligned}
\right.
\end{eqnarray}

By the definition of $\cal H,$
it follows that

\begin{eqnarray}\label{eq:27}
  \begin{split}
   \mathbb E^{\bar u}\bigg[\int_0^T (l^u(t)-\bar l(t))dt\bigg]=&{\mathbb E} ^{\bar u}
\bigg[\int_0^T \bigg ( { \cal H}(t,\Theta^u(t),u(t), \mathbb E[\Theta^u(t)], \mathbb E[u(t)],
 \bar\Lambda(t),\bar R_2(t))
  - {\bar{\cal H}} (t)\bigg)dt\bigg]
\\&- \mathbb E^{\bar u}\bigg[\int_0^T \bigg (\langle \bar p(t), b^u (t) - \bar b (t)\rangle + \langle \bar q_1(t), \sigma_1^u(t) - \bar \sigma_1 (t)\rangle
+\langle \bar  q_2(t), \sigma _2^u(t) - {\bar \sigma}_2 (t)
\rangle
\\&+ \langle \bar k(t),  f^u(t) - \bar f (t)\rangle+\langle \bar R_2(t)-\bar \sigma_2^*(t)\bar p(t)-\bar z_2 ^*(t)\bar k(t),h^u(t)-\bar h(t)\rangle \bigg)dt\bigg]
  \end{split}
\end{eqnarray}

 By using It\^{o} formula to $\la \bar p(t), x^u(t)-\bar x(t)\ra
+\la \bar k(t), y^u(t)-\bar y(t)\ra$ and
 taking expectation with respect to $P^{\bar u}$,
  we obtain that

\begin{eqnarray}
   && \mathbb E^{\bar u} \bigg[ \la \bar\Phi_x(T)
+\frac{1}{\bar \rho(T)}\mathbb E^{\bar u} \left[\bar\Phi_{x¡¯}(T)\right]
-\big[\bar \phi_x^*(T)\bar k(T)
+\frac{1}{\bar \rho(T)}\mathbb E^{\bar u} \left[\bar \phi_{x'}^*(T)\bar k(T)\right]\big], x^u(T)-\bar x(T)\ra\bigg]+\mathbb E^{\bar u} \big[ \la \bar k(T), \phi^u(T)-\bar \phi(T)\ra\big]\nonumber
\\=&&\mathbb E^{\bar u}\bigg[\int_0^T \bigg (\langle \bar p(t), b^u (t) - \bar b (t)\rangle + \langle \bar q_1(t), \sigma_1^u (t) - \bar \sigma_1 (t)\rangle
+ \langle \bar q_2(t), \sigma _2^u(t) - {\bar \sigma}_2 (t)\rangle
\nonumber
\\&&+ \langle \bar k(t), f^u(t) - \bar f (t)\rangle
+\langle \bar p(t), \sigma_2^u(t)(\bar h(t)- h^u(t))\rangle +\langle \bar k(t), z_2^u(t)(\bar h(t)- h^u(t))\rangle \bigg)dt
\bigg]
\nonumber
\\&&-\mathbb E^{\bar u} \bigg[\int_0^T \la { \bar{\cal H}}_x (t)+\frac{1}{\bar\rho(t)}\mathbb E^{\bar u}[{\cal H}_{x'} (t)], x^u (t) - \bar
x (t)\ra dt\bigg]
\nonumber
\\&&-\mathbb E^{\bar u} \bigg[\int_0^T \la { \bar{\cal H}}_y (t)+\frac{1}{\bar\rho(t)}\mathbb E^{\bar u}[{\cal H}_{y'} (t)], y^u (t) - \bar
y (t)\ra dt\bigg]
\nonumber
\\&&-\mathbb E^{\bar u}
 \bigg[\int_0^T \la { \bar{\cal H}}_{z_1} (t)+\frac{1}{\bar\rho(t)}\mathbb E^{\bar u}[{\cal H}_{z_1'} (t)], z_1^u (t) - \bar
z_1 (t)\ra dt\bigg]
\nonumber
\\&&-\mathbb E^{\bar u} \bigg[\int_0^T \la { \bar{\cal H}}_{z_2} (t)+\frac{1}{\bar\rho(t)}\mathbb E^{\bar u}[{\cal H}_{z_2'} (t)], z_2^u (t) - \bar
z_2 (t)\ra dt\bigg]
\nonumber
\\&&-{\mathbb E}
 \big [  \left <
y^u (0)
- \bar y (0), \bar \gamma_y (0)  \right
> \big ]
\end{eqnarray}

which implies that

\begin{eqnarray} \label{eq:29}
   &&\mathbb E^{\bar u}\bigg[\int_0^T \bigg (\langle \bar p(t), b^u (t) - \bar b (t)\rangle + \langle \bar q_1(t), \sigma_1^u (t) - \bar \sigma_1 (t)\rangle
+ \langle \bar q_2(t), \sigma _2^u(t) - {\bar \sigma}_2 (t)\rangle+ \langle \bar k(t), f^u(t) - \bar f (t)\rangle \bigg)dt\bigg]\nonumber
\\=&& \mathbb E^{\bar u} \big[ \la \bar\Phi_x(T)
+\frac{1}{\bar \rho(T)}\mathbb E^{\bar u} [\bar\Phi_x(T)] , x^u(T)-\bar x(T)\ra\big]
\\&&+\mathbb E^{\bar u} \bigg[ \la \bar k(T), \phi^u(T)-\bar \phi(T)-\bar\phi_x(T)(x^u(T)-\bar x(T))-\bar\phi_x(T)(\mathbb E[x^u(T)]
-\mathbb E[\bar x(T)])\ra\bigg]\nonumber
\\&&+ \mathbb E\bigg[\int_0^T\bigg(\langle \bar p(t), \sigma_2^u(t)(h^u(t)-\bar h(t))\rangle +\langle \bar k(t), z_2^u(t)(h^u(t)-\bar h(t))\rangle \bigg)dt\bigg]\nonumber
\\&&+\mathbb E^{\bar u} \bigg[\int_0^T \la { \bar{\cal H}}_x (t)+\frac{1}{\bar\rho(t)}\mathbb E^{\bar u}[{\cal H}_{x'} (t)], x^u (t) - \bar
x (t)\ra dt\bigg]
\nonumber
\\&&+\mathbb E^{\bar u} \bigg[\int_0^T \la { \bar{\cal H}}_y (t)+\frac{1}{\bar\rho(t)}\mathbb E^{\bar u}[{\cal H}_{y'} (t)], y^u (t) - \bar
y (t)\ra dt\bigg]
\nonumber
\\&&+\mathbb E^{\bar u}
 \bigg[\int_0^T \la { \bar{\cal H}}_{z_1} (t)+\frac{1}{\bar\rho(t)}\mathbb E^{\bar u}[{\cal H}_{z_1'} (t)], z_1^u (t) - \bar
z_1 (t)\ra dt\bigg]
\nonumber
\\&&+\mathbb E^{\bar u} \bigg[\int_0^T \la { \bar{\cal H}}_{z_2} (t)+\frac{1}{\bar\rho(t)}\mathbb E^{\bar u}[{\cal H}_{z_2¡®} (t)], z_2^u (t) - \bar
z_2 (t)\ra dt\bigg]\nonumber
\\&&+{\mathbb E}
 \big [  \left <
y^u (0)
- \bar y (0), \bar \gamma_y (0)  \right
> \big ]
\bigg]
\end{eqnarray}

Applying It\^{o} formula to
$(\rho^u(t)-\bar \rho(t)) \bar r(t), $
yields that
\begin{eqnarray}
  \begin{split}
    \mathbb E [ (\rho^u(T)-\bar\rho(T))\bar \Phi(T)]= &-\mathbb E\bigg[\int_0^T(\rho^u(t)-\bar\rho(t))(\bar l(t)+\bar R_2(t)\bar h(t))dt\bigg]
    +\mathbb E\bigg[\int_0^T\bar R_2(t)(\rho^u(t)h^u(t)-\bar\rho(t)\bar  h(t))dt\bigg],
  \end{split}
\end{eqnarray}

which implies that

\begin{eqnarray}\label{eq:31}
  \begin{split}
    \mathbb E  [(\rho^u(T)-\bar\rho(T))\bar \Phi(T)]+\mathbb E\bigg[\int_0^T(\rho^u(t)
    -\bar\rho(t))\bar l(t)dt\bigg]=
    \mathbb E\bigg[\int_0^T\bar R_2(t)\rho^u(t)(h^u(t)-\bar  h(t))dt\bigg]
  \end{split}
\end{eqnarray}

Putting \eqref{eq:29} into \eqref{eq:27},
 it follows that
\begin{eqnarray}\label{eq:32}
  \begin{split}
   \mathbb E^{\bar u}
   \bigg[\int_0^T
    (l^u(t)-\bar l(t))dt\bigg]={\mathbb E} ^{\bar u}
\bigg[\int_0^T \bigg ( &{ \cal H}(t,\Theta^u(t), \mathbb E[\Theta^u(t)], u(t),\mathbb E[u(t)],
 \bar\Lambda(t),\bar R_2(t))
  - {\bar{\cal H}} (t)
\\& - \big < \bar{\cal H}_{x} (t)+\frac{1}{\bar\rho(t)}\mathbb E^{\bar u}[{\cal H}_{x'} (t)],x^u (t) - \bar
x (t)\big>
\\&- \big <{\bar{\cal H}}_y (t)+\frac{1}{\bar\rho(t)}\mathbb E^{\bar u}[{\cal H}_{y'} (t)],
y^u (t) - \bar
y (t)\big>
\\&- \big <\bar{\cal H}_{z_1} (t)+\frac{1}{\bar\rho(t)}\mathbb E^{\bar u}[{\cal H}_{z_1'} (t)],
 z^u_1 (t) - \bar
z_1 (t)\big>
\\&- \big <\bar{\cal H}_{z_2} (t)+\frac{1}{\bar\rho(t)}\mathbb E^{\bar u}[{\cal H}_{z_2'} (t)],
 z^u_2 (t) - \bar
z_2 (t)\big>
\\&-
 \langle (\sigma_2^u(t)-\bar{\sigma}_2(t))
 (h^u(t)-\bar
 h(t)), \bar p(t)\rangle
 \\&-
 \langle (z_2^u(t)-\bar{z}_2(t))
 (h^u(t)-\bar
 h(t)), \bar k(t)\rangle \bigg)dt\bigg] \\
&  -\mathbb E^{\bar u} \big[ \la \bar\Phi_x(T)
+\frac{1}{\bar \rho(T)}\mathbb E^{\bar u} [\bar\Phi_x(T)] , x^u(T)-\bar x(T)\ra\big]
\\&-\mathbb E^{\bar u} \bigg[ \la \bar k(T), \phi^u(T)-\bar \phi(T)-\bar\phi_x(T)(x^u(T)-\bar x(T))-\bar\phi_x(T)(\mathbb E[x^u(T)]
-\mathbb E[\bar x(T)])\ra\bigg]
\\&-{\mathbb E}
 \big [  \left <
y^u (0)
- \bar y (0), \bar \gamma_y (0)  \right
> \big ]
  \end{split}
  \end{eqnarray}

Then by putting\eqref{eq:31} and
 \eqref{eq:32} into \eqref{eq:26}, we
  obtain \eqref{eq:21}. The proof is complete.
 \end{proof}

Since the control domain $U$ is convex,
 for any given admissible controls
 $u (\cdot) \in {\cal A}$,
the following perturbed control
 process $u^\epsilon (\cdot)$:
\begin{eqnarray*}
u^\epsilon (\cdot) := \bar u (\cdot) + \epsilon ( u (\cdot) - \bar u (\cdot) ) , \quad  0 \leq \epsilon \leq 1 ,
\end{eqnarray*}
is also in ${\cal A}$. We denote by $(\bar \Theta(\cdot),\bar\rho(\cdot))=(\bar x (\cdot),
\bar y (\cdot),
\bar z_1 (\cdot),\bar z_2(\cdot),\bar\rho(\cdot))$ and
$(\Theta^\epsilon (\cdot), \rho^\epsilon (\cdot))=(x^\epsilon (\cdot), y^\epsilon (\cdot), z^\epsilon_1 (\cdot), z^\eps_2(\cdot),\rho^\eps)$ the corresponding state processes
associated with $\bar u (\cdot)$ and $u^\epsilon (\cdot)$, respectively. Denote by $(\bar \Lambda(\cdot),\bar \Gamma(\cdot))=(\bar p(\cdot),\bar q_1(\cdot), \bar { q}_2(\cdot),\bar k(\cdot),\bar r(\cdot),\bar R_1(\cdot),\bar R_2(\cdot ) )$ the adjoint process
associated with the admissible pair $(\bar u (\cdot); \bar x (\cdot),
\bar y(\cdot), \bar z_1(\cdot),
 \bar z_2(\cdot))$.

\begin{lemma} \label{lem:3.5}
  Suppose  Assumptions \ref{ass:1.1}
   and \ref{ass:1.2} hold. Then
\begin{eqnarray*}
&&{\mathbb E} \bigg [ \sup_{t\in \cal T} | x^\epsilon (t) - \bar
x (t) |^4  \bigg ]
+{\mathbb E} \bigg [ \sup_{t\in \cal T} | y^\epsilon (t) - \bar
y (t) |^4  \bigg ]
+{\mathbb E} \bigg [
\bigg(\int_0^T| z^\epsilon_1 (t) - \bar
z _1(t) |^2 dt\bigg)^{2} \bigg ]
+{\mathbb E} \bigg [
\bigg(\int_0^T| z^\epsilon_2 (t) - \bar
z _2(t) |^2 dt\bigg)^{2} \bigg ]= O (\epsilon^4) \ .
\end{eqnarray*}
and
\begin{eqnarray*}
&&{\mathbb E} \bigg [ \sup_{t\in \cal T} | \rho^\epsilon (t) - \bar
\rho (t) |^2  \bigg ]
= O (\epsilon^2) \ .
\end{eqnarray*}
\end{lemma}
\begin{proof}
  The proof can be obtained
  directly by Lemmas \ref{lem:3.3}
  and \ref{lem:3.4}.
\end{proof}

Now we are in the position to apply
  Lemma \ref{lem4} and Lemma \ref{lem:3.5} to derive the variational formula for the
  cost functional
$J(u(\cdot))$ in terms of the Hamiltonian ${\cal H}$.

\begin{theorem}\label{them:3.1}
 Suppose that Assumptions \ref{ass:1.1} and
 \ref{ass:1.2} holds.
Then for any admissible control
 $u (\cdot) \in {\cal A}$,
 a variation formula  for the
 cost functional $J (u (\cdot))$ at $\bar u (\cdot)$  is given by
\begin{eqnarray}\label{eq:4.4}
&& \frac{d}{d\epsilon} J ( \bar u (\cdot) + \epsilon ( u (\cdot) - \bar u (\cdot) ) ) |_{\epsilon=0} \nonumber \\
&& := \lim_{\epsilon \rightarrow 0^+}
\frac{ J ( \bar u (\cdot) + \epsilon ( u (\cdot) - \bar u (\cdot) ) )
- J( \bar u (\cdot) ) }{\epsilon} \nonumber \\
&& = {\mathbb E} \bigg
 [ \int_0^T \left < \bar \rho(t){\bar{\cal H}}_u (t)+
\mathbb E^{\bar u}[{\bar{\cal H}}_u (t)],
u (t) -
 \bar u (t) \right > d t \bigg ] .
\end{eqnarray}
\end{theorem}

\begin{proof}
To simplify our notations, denote by
\begin{eqnarray}
\beta^\epsilon &:=& {\mathbb E} ^{\bar u}
\bigg[\int_0^T \bigg \{ { \cal H}(t,\Theta^{u^\eps}(t), \mathbb E[\Theta^{u^\eps}(t)], u(t),\mathbb E[u(t)],
 \bar\Lambda(t),\bar R_2(t)) - {\bar{\cal H}} (t)\nonumber
   \\&&\quad\quad\quad- \big < {\bar{\cal H}_x} (t)+\frac{1}{\bar\rho(t)}\mathbb E^{\bar u}[{\cal H}_{x'} (t)],x^{u^\eps} (t) - \bar
x (t)
\big>\nonumber
\\&&\quad\quad\quad- \big <{ \bar{\cal H}}_y (t)+\frac{1}{\bar\rho(t)}\mathbb E^{\bar u}[{\cal H}_{y'} (t)],
y^{u^\eps} (t) - \bar
y (t)\big>
\\&&\quad\quad\quad- \big <{ \bar{\cal H}}_{z_1} (t)+\frac{1}{\bar\rho(t)}\mathbb E^{\bar u}[{\cal H}_{z_1'} (t)],
z^{u^\eps}_1 (t) - \bar
z_1 (t)\big>\nonumber
\\&&\quad\quad\quad- \big <{ \bar{\cal H}}_{z_2} (t)+\frac{1}{\bar\rho(t)}\mathbb E^{\bar u}[{\cal H}_{z_2¡®} (t)],
z^{u^\eps}_2 (t) - \bar
z_2 (t)\big>
\nonumber
\\&&\quad\quad\quad- \big <{ \bar{\cal H}}_{u} (t)+\frac{1}{\bar\rho(t)}\mathbb E^{\bar u}[{\cal H}_{u} (t)],u (t) - \bar u(t)\big>
\nonumber
\\&&\quad\quad\quad-
 \langle (\sigma_2^{u^\eps}(t)-\bar{\sigma}_2(t))
 (h^{u^\eps}(t)-\bar
 h(t)), \bar p(t)\rangle\nonumber
 \\&&\quad\quad\quad-
 \langle (z_2^{u^\eps}(t)-\bar{z}_2(t))
 (h^{u^\eps}(t)-\bar
 h(t)), \bar k(t)\rangle\bigg\} dt\bigg]
\nonumber \\
&&\quad\quad\quad + {\mathbb E}^{\bar u} \big [ \Phi^{u^\eps}(T)
 - \bar \Phi (T) -  \la
x^{u^\eps} (T) - \bar x (T), \bar\Phi_x(T)+\frac{1}{\bar \rho(T)}\mathbb E^{\bar u} \left[\bar\Phi_x(T)\right]
\ra \big ]
\nonumber \\
&&\quad\quad\quad - {\mathbb E}^{\bar u}
\big [ \la \phi^{u^\eps} (T) - \bar \phi (T), \bar k(T)\ra - \langle
\bar \phi_x^*(T)\bar k(T)+\frac{1}{\bar \rho(T)}\mathbb E^{\bar u} \left[\bar\phi_x^*(T)\bar k(T)\right],
x^{u^\eps} (T) - \bar x (T) \rangle \big ]
\nonumber \\
&&\quad\quad\quad + {\mathbb E}
 \big [ \gamma^{u^\eps}(0) - \bar \gamma (0)
 - \la
y^{u^\eps} (0)
- \bar y (0), \bar \gamma_y (0)  \ra\big ]
\nonumber
\\&&\quad\quad\quad+\mathbb E\bigg[\int_0^T \bar R_2(t)(\rho^{u^\eps}(t)-\bar\rho(t))(h^{u^\eps}(t)-\bar
 h(t))dt\bigg] \nonumber
\\&&\quad\quad\quad+\mathbb E\bigg[\int_0^T(l^{u^\eps}(t)-\bar l(t))( \rho^{u^\eps}(t)-\bar
\rho(t))dt\bigg]\nonumber
\\&&\quad\quad\quad+\mathbb E \bigg[ ( \rho^{u^\eps} (T) - \bar \rho(T)) (\Phi^{u^\eps} (T) -
\bar \Phi (T))\bigg].
\end{eqnarray}
In view of  Lemma \ref{lem4}, we get
\begin{eqnarray}\label{eq:4.13}
J ( u^\epsilon (\cdot) ) - J ( \bar u (\cdot) )
= \beta^\epsilon + \epsilon {\mathbb E} \bigg
 [ \int_0^T \left < \bar \rho(t){\bar{\cal H}}_u (t)+
\mathbb E^{\bar u}[{\bar{\cal H}}_u (t)],
u (t) -
 \bar u (t) \right > d t  \bigg ] .
\end{eqnarray}
By Assumptions \ref{ass:1.1} and
\ref{ass:1.2}, combining the Taylor Expansions, Lemma \ref{lem:3.5}, and the dominated
convergence theorem, we obtain that
\begin{eqnarray}\label{eq:4.121}
\beta^\epsilon = o(\epsilon) .
\end{eqnarray}
Putting  \eqref{eq:4.121} into \eqref{eq:4.13} gives
\begin{eqnarray*}
\lim_{\epsilon \rightarrow 0^+} \frac{J (u^\epsilon (\cdot) ) - J (\bar u (\cdot))}{\epsilon}
=  {\mathbb E} \bigg
 [ \int_0^T \left < \bar \rho(t){\bar{\cal H}}_u (t)+
\mathbb E^{\bar u}[{\bar{\cal H}}_u (t)],
u (t) -
 \bar u (t) \right > d t \bigg ] .
\end{eqnarray*}
This completes the proof.
\end{proof}

\section{Main Results}
 This section is devoted to
   establishing the necessary condition and sufficient maximum principles for Problem \ref{pro:1.1} or \ref{pro:1.2}.
We first prove the necessary condition of optimality for the existence of an optimal control.

\begin{theorem}[{\bf Necessary Stochastic Maximum principle}]
 Suppose that  Assumptions \ref{ass:1.1}
  and \ref{ass:1.2} holds and $( \bar u (\cdot); \bar x (\cdot), \bar y(\cdot), \bar z_1(\cdot),
  \\\bar z_2(\cdot),\bar\rho(\cdot) )$ is
an optimal pair of  Problem \ref{pro:1.2}.
 Then
\begin{eqnarray}\label{eq:4.15}
\left <  \mathbb E[\bar \rho(t){\bar{\cal H}}_u (t)|\mathscr F^Y_t]+
\mathbb E^{\bar u}[{\bar{\cal H}}_u (t)],
 u- \bar u (t) \right > \geq 0 , \quad \forall u \in U ,\ a.e.\ a.s..
\end{eqnarray}
\end{theorem}

\begin{proof}
Because  all admissible controls
are
$\{\mathscr F^Y_t\}_{t\in \cal T}$-adapted
processes, using  the property of conditional
expectation, Theorem \ref{them:3.1} and the optimality of $\bar u(\cdot)$,  we
obtain  that
\begin{eqnarray*}
&& {\mathbb E} \bigg [ \int_0^T
 \langle \mathbb E[\bar \rho(t){\bar{\cal H}}_u (t)|\mathscr F^Y_t]+
\mathbb E^{\bar u}[{\bar{\cal H}}_u (t)] ,
u (t) - \bar u (t) \rangle  d t \bigg ] \\
&&=
{\mathbb E} \bigg [ \int_0^T
 \langle \mathbb E[\bar \rho(t){\bar{\cal H}}_u (t)+
\mathbb E^{\bar u}[{\bar{\cal H}}_u (t)]|\mathscr F^Y_t] ,
u (t) - \bar u (t) \rangle  d t \bigg ]
\\&&={\mathbb E} \bigg
 [ \int_0^T \left < \bar \rho(t){\bar{\cal H}}_u (t)+
\mathbb E^{\bar u}[{\bar{\cal H}}_u (t)],
u (t) -
 \bar u (t) \right > d t \bigg ]  \\
&& = \lim_{\epsilon \rightarrow 0^+} \frac{J( \bar u (\cdot) + \epsilon (
u (\cdot) - \bar u (\cdot) ) ) - J ( \bar u (\cdot) )}{\epsilon} \geq 0 ,
\end{eqnarray*}
which proves
\eqref{eq:4.15} is satisfied.
  The proof is complete.

\end{proof}

Next we prove the sufficient condition of optimality for the existence of an optimal control of
Problem \ref{pro:1.2} in the case when
the observation process does not
contains the control process and the
state process.
Assume that
$$h(t,x,u)=h(t)$$ is an
$\mathscr F^Y_t-$ adapted bounded
process. Introduce a new probability measure $\mathbb Q$ on $(\Omega, \mathscr F)$ by
\begin{eqnarray}
  d\mathbb Q=\rho(1)d\mathbb P,
\end{eqnarray}
where
\begin{eqnarray} \label{eq:43}
  \left\{
\begin{aligned}
  d \rho(t)=&  \rho(t) h(s)dY(s)\\
  \rho(0)=&1.
\end{aligned}
\right.
\end{eqnarray}

\begin{theorem}{\bf [Sufficient Maximum Principle] } \label{thm:4.1}
 Suppose that Assumptions \ref{ass:1.1}
  and \ref{ass:1.2}  hold and  $(\bar u (\cdot);
  \bar \Theta (\cdot))=(\bar u (\cdot);
  \bar x (\cdot),
  \bar y(\cdot), \bar z_1(\cdot)
  ,\bar z_2 (\cdot))$ is an admissible
   pair with $\phi(x)=\phi x,$
  where $\phi$ is $\mathscr F_T-$measurable bounded  random variable. Assume that
\begin{enumerate}
\item[(i)]  $\Phi$ and $\gamma$ is convex in $x$ and $y,$ respectively,
\item[(ii)] the Hamiltonian ${\cal H}$ is convex in $(x, y, z_1, z_2,  u)$,
\item[(iii)]
\begin{eqnarray*}\label{eq:5.119}
&& \mathbb E^Q\bigg[{\cal H} ( t,\bar\Theta(t),\bar u(t), \mathbb E[\bar \Theta(t)],\mathbb E[\bar u(t)),
 \bar\Lambda(t), \bar R_2(t)) |\mathscr F^Y_t\bigg]\nonumber \\
&& = \min_{(u, u') \in U\times U }
 \mathbb E^Q\bigg[{\cal H} ( t,\bar\Theta(t), u,\mathbb E[\bar \Theta(t)],u',
 \bar\Lambda(t), \bar R_2(t)
 |\mathscr F^Y_t\bigg], \quad \mbox {a.e.\ a.s.} ,
\end{eqnarray*}
\end{enumerate}
where $( {\bar\Lambda}(\cdot),
 {\bar\Gamma}(\cdot))=( \bar k(\cdot),
 \bar p(\cdot), \bar q_1(\cdot),
  \bar q_2(\cdot), \bar r(\cdot),
   \bar R_1(\cdot), \bar R_2(\cdot) )$ is the
adjoint process associated with $( \bar u(\cdot);  \bar\Theta(\cdot)).$
Then $(\bar u (\cdot),  \bar \Theta (\cdot))$ is an optimal pair of Problem \ref{pro:1.2}.
\end{theorem}

\begin{proof}
Let $({u}(\cdot);x^u(\cdot),
 y^u(\cdot),
 z^u_1(\cdot),  z^u_2(\cdot),
\rho^u(\cdot))$
 be an arbitrary admissible pair. From Lemma \ref{lem4},
the difference $J ( u (\cdot) ) - J ( \bar u (\cdot) )$ can be represented as follows
\begin{eqnarray}\label{eq:40}
&&J (u (\cdot)) - J (\bar u(\cdot))\nonumber
\\ &=& {\mathbb E}^{Q}
\bigg[\int_0^T \bigg \{ { \cal H}(t,\Theta^u(t), \mathbb E[\Theta^u(t)], u(t),\mathbb E[u(t)],
 \bar\Lambda(t),\bar R_2(t)) - {\bar{\cal H}} (t)\nonumber
   \\&&\quad\quad\quad- \big < {\bar{\cal H}_x} (t)+\frac{1}{\bar\rho(t)}\mathbb E^{\bar u}[{\cal H}_{x'} (t)],x^u (t) - \bar
x (t)
\big>\nonumber
- \big <{ \bar{\cal H}}_y (t)+\frac{1}{\bar\rho(t)}\mathbb E^{\bar u}[{\cal H}_{y'} (t)],
y^u (t) - \bar
y (t)\big>
\\&&\quad\quad\quad- \big <{ \bar{\cal H}}_{z_1} (t)+\frac{1}{\bar\rho(t)}\mathbb E^{\bar u}[{\cal H}_{z_1'} (t)],
z^u_1 (t) - \bar
z_1 (t)\big>- \big <{ \bar{\cal H}}_{z_2} (t)+\frac{1}{\bar\rho(t)}\mathbb E^{\bar u}[{\cal H}_{z_2'} (t)],
z^u_2 (t) - \bar
z_2 (t)\big>\bigg\} dt\bigg]
\nonumber \\
&&\quad\quad\quad + {\mathbb E}^{\bar u} \big [ \Phi^u(T)
 - \bar \Phi (T) -  \la
x^u (T) - \bar x (T), \bar\Phi_x(T)+\frac{1}{\bar \rho(T)}\mathbb E^{\bar u} \left[\bar\Phi_x(T)\right]
\ra \big ]
\nonumber \\
&&\quad\quad\quad + {\mathbb E}
 \big [ \gamma^u(0) - \bar \gamma (0)
 - \la
y^u (0)
- \bar y (0), \bar \gamma_y (0)  \ra\big ]
\end{eqnarray}

In view of  the convexity of ${\cal H}$, $\Phi$
and $\gamma$ (i.e. Conditions (i) and (ii)),
we obtain
\begin{eqnarray}\label{eq:41}
\begin{split}
 &{ \cal H}(t,\Theta^u(t), \mathbb E[\Theta^u(t)], u(t),\mathbb E[u(t)],
 \bar\Lambda(t),\bar R_2(t)) - {\bar{\cal H}} (t)\nonumber
\\ \geq&  \big < {\bar{\cal H}}_x (t),x^u (t) - \bar
x (t)
\big>+
\big < {\bar{\cal H}}_{x'}(t),\mathbb E[x^u (t)] -\mathbb E[ \bar x (t)]\big>
\nonumber
\\&
+ \big <\bar{ \cal H}_y (t),
y^u (t) - \bar
y (t)\big>+
\big < \bar{\cal H}_{y'} (t),\mathbb E[y^u (t)] -\mathbb E[ \bar y (t)]\big>
\\&
+ \big <\bar{ \cal H}_{z_1} (t),
z^u_1 (t) - \bar
z_1 (t)\big>+
\big < \bar{\cal H}_{z'_1} (t),\mathbb E[z^u_1 (t)] -\mathbb E[ \bar z_1 (t)]\big>
\\&
+ \big <\bar{ \cal H}_{z_2} (t),
z^u_2 (t) - \bar
z_2 (t)\big>+
\big < \bar{\cal H}_{z'_2} (t),\mathbb E[z^u_2 (t)] -\mathbb E[ \bar z_2 (t)]\big>
\\&
+ \big <\bar{ \cal H}_{u} (t),u(t) - \bar u (t)\big>+
\big < \bar{\cal H}_{u'} (t),\mathbb E[u (t)] -\mathbb E[ \bar u (t)]\big>,
\end{split}
\end{eqnarray}

\begin{eqnarray}\label{eq:42}
 \Phi^u(T) - \bar \Phi (T) \geq \left < X^u (T) - \bar X (T), \bar\Phi_x (T) \right >+\left < \mathbb E[X^u (T)] - \mathbb E[\bar X (T)], \bar\Phi_x (T) \right >
\end{eqnarray}
and
\begin{eqnarray}\label{eq:43}
 \gamma^u(0) - \bar \gamma (0) \geq \left < y^u (0) - \bar y (0), \gamma_y (0) \right >.
\end{eqnarray}

Furthermore,
 by the convex optimization principle
 (see Proposition 2.21 of \cite{ET1976}) and
  the optimality condition (iii),
   we obtain that
\begin{eqnarray}\label{eq:5.5}
 \langle u (t) - \bar u (t), \mathbb E^Q\big[{\cal H}_{u} ( t) |\mathscr F^Y_t\big] \rangle
 +\langle \mathbb E[u (t)] -
 \mathbb E[\bar u (t)], \mathbb E^Q\big[{\cal H}_{u'} ( t) |\mathscr F^Y_t\big] \rangle
\geq 0 .
\end{eqnarray}
which imply that

\begin{eqnarray}\label{eq:45}
\mathbb E^Q\big[\langle u (t) - \bar u (t), {\cal H}_{u} ( t)  \rangle\big]
 +\mathbb E^Q\big[\langle \mathbb E[u (t)] -\mathbb E[\bar u (t)],{\cal H}_{u'} ( t)\rangle\big]
\geq 0.
\end{eqnarray}

 Inserting \eqref{eq:41},\eqref{eq:42},\eqref{eq:43} and  \eqref{eq:45} into \eqref{eq:40},
we get
\begin{eqnarray}
J (u (\cdot)) - J (\bar u (\cdot)) \geq 0 .
\end{eqnarray}
 Since $u (\cdot)$ is arbitrary,
 we get that $\bar u (\cdot)$ is
 an optimal control process
and thus $(\bar u (\cdot), \bar x (\cdot), \bar y(\cdot), \bar z_1(\cdot), z_2(\cdot))$ is an optimal pair. The proof is completed.
\end{proof}

\bibliographystyle{model1a-num-names}

\begin{thebibliography}{00}

\bibitem{Anto} Antonelli, F. (1993). Backward-forward stochastic differential equations.
     {\it The Annals of Applied Probability,} 777-793.





\bibitem {DuEp} Duffie, D., \& Epstein, L. G. (1992). Asset pricing with stochastic differential utility.
    {\it The Review of Financial Studies,} 5(3), 411-436.





\bibitem{ElPe} El Karoui, N., Peng, S., \& Quenez, M. C. (1997). Backward stochastic differential equations in finance. {\it Mathematical finance,} 7(1), 1-71.


\bibitem{Xu95}
Hu, M. (2017). Stochastic global maximum principle for optimization with recursive utilities.
{\it Probability, Uncertainty and Quantitative Risk,} 2(1), 1.




\bibitem{Mou}Mou, L., \& Yong, J. (2007). A variational formula for stochastic controls and some applications.
     {\it Pure and Applied Mathematics Quarterly,} 3(2), 539-567.









\bibitem{PaPe90} Pardoux, E., \& Peng, S. (1990). Adapted solution of a backward stochastic differential equation. {\it Systems \& Control Letters,} 14(1), 55-61.

   \bibitem{ShWz2006}
Shi, J., \& Wu, Z. (2006). The Maximum I Principle for Fully Coupled Forward-backward Stochastic Control System.
{\it Acta Automatica Sinica,} 32(2), 161.





   \bibitem{Wu9}
Wu, Z. (2013). A general maximum principle for optimal control of forward-backward stochastic systems.
{\it Automatica,} 49(5), 1473-1480.


\bibitem{Yong10}
Yong, J. (2010). Optimality variational principle for controlled forward-backward stochastic differential equations with mixed initial-terminal conditions.
{\it SIAM  Journal on Control and Optimization,} 48(6), 4119-4156.



\bibitem{K1956} M. Kac, Foundations of kinetic theory,
Proceedings of the 3rd Berkeley Symposium on Mathematical Statistics and Probability 3 (1956) 171-197.

\bibitem{M1966} H.P. McKean, A class of Markov processes associated with nonlinear parabolic equations,
Proceedings of the National Academy of Sciences 56 (1966) 1907-1911.

\bibitem{AD2011} D. Andersson, B. Djehiche, A maximum principle for SDEs of mean-field type,
Applied Mathematics and Optimization 63 (2011) 341-356.

\bibitem{BDL2011} R. Buckdahn, B. Djehiche, J. Li, A general stochastic maximum principle for SDEs of
mean-field type, Applied Mathematics and Optimization 64 (2011) 197-216.

\bibitem{L2012} J. Li, Stochastic maximum principle in the mean-field controls,
Automatica 48 (2012) 366-373.

\bibitem{MOZ2012} T. Meyer-Brandis, B. {\O}ksendal, X.Y. Zhou, A mean-field stochastic maximum principle
via Malliavin calculus, Stochastics 84 (2012) 643-666.

\bibitem{SS2013} Y. Shen, T.K. Siu, The maximum principle for a jump-diffusion mean-field model
and its application to the mean-variance problem, Nonlinear Analysis: Theory, Methods \& Applications
86 (2013) 58-73.

\bibitem{DHQ2013} H. Du, J. Huang, Y. Qin, A stochastic maximum principle for delayed mean-field
stochastic differential equations and its applications, IEEE Transactions on Automatic Control 58 (2013), 3212-3217.

\bibitem{SMS2014} Y. Shen, Q. Meng, P. Shi, Maximum principle for mean-field jump-diffusion stochastic delay
differential equations and its application to finance, Automatica 50 (2014) 1565-1579.


\end{thebibliography}

\end{document}